\documentclass[11pt]{amsart}
\usepackage[utf8]{inputenc}
\usepackage[margin=1in]{geometry}
\usepackage{hyperref}
\usepackage{amsfonts}
\usepackage{amsmath}
\usepackage{amsthm}
\usepackage{amssymb}
\usepackage{mathrsfs}
\usepackage{tikz}
\usepackage[nameinlink,noabbrev]{cleveref}

\DeclareFontFamily{OT1}{pzc}{}
\DeclareFontShape{OT1}{pzc}{m}{it}{<-> s * [1.200] pzcmi7t}{}
\DeclareMathAlphabet{\mathpzc}{OT1}{pzc}{m}{it}

\numberwithin{equation}{section}

\newcommand{\R}{\mathbb{R}}
\newcommand{\N}{\mathbb{N}}
\renewcommand{\H}{\mathfrak{H}}

\newcommand{\C}{\mathbb{C}}

\newcommand{\B}{\mathcal{B}}
\newcommand{\Os}{\mathcal{O}}
\newcommand{\F}{\mathfrak{F}}
\renewcommand{\a}{\mathpzc{a}}
\renewcommand{\b}{\mathpzc{b}}
\newcommand{\wt}[1]{\widetilde{#1}}

\newcommand{\dom}{\operatorname{dom}}

\renewcommand{\d}[2]{\frac{#1}{#1 #2}}

\theoremstyle{plain}
\newtheorem{theorem}{Theorem}[section]

\newtheorem{corollary}[theorem]{Corollary}
\newtheorem{proposition}[theorem]{Proposition}

\theoremstyle{definition}
\newtheorem{definition}[theorem]{Definition}

\theoremstyle{remark}
\newtheorem{remark}{Remark}

\title[Segal-Bargmann Transforms Associated to a Family of Coupled Supersymmetries]{Segal-Bargmann Transforms Associated to \\ a Family of Coupled Supersymmetries}
\author{Cameron L. Williams}
\address{
Department of Mathematics \\
Embry-Riddle Aeronautical University \\
3700 Willow Creek Road \\
Prescott, AZ 86301
}

\begin{document}

\begin{abstract}
The Segal-Bargmann transform is a Lie algebra and Hilbert space isomorphism between real and complex representations of the oscillator algebra. The Segal-Bargmann transform is useful in time-frequency analysis as it is closely related to the short-time Fourier transform. The Segal-Bargmann space provides a useful example of a reproducing kernel Hilbert space. Coupled supersymmetries (coupled SUSYs) are generalizations of the quantum harmonic oscillator that have a built-in supersymmetric nature and enjoy similar properties to the quantum harmonic oscillator. In this paper, we will develop Segal-Bargmann transforms for a specific class of coupled SUSYs which includes the quantum harmonic oscillator as a special case. We will show that the associated Segal-Bargmann spaces are distinct from the usual Segal-Bargmann space: their associated weight functions are no longer Gaussian and are spanned by stricter subsets of the holomorphic polynomials. The coupled SUSY Segal-Bargmann spaces provide new examples of reproducing kernel Hilbert spaces.
\end{abstract}

\maketitle

\section{Introduction}

The harmonic oscillator is one of the most fundamental systems in quantum mechanics and quantum field theory due in part to its Lie algebraic structure and the fact that many potentials near their local minima may be well-approximated by a quadratic potential. The Lie algebra associated to the quantum harmonic oscillator is comprised of the Hamiltonian and creation and annihilation operators which have the interpretations of creating or destroying quanta of the quantum field. For our purposes, we define the quantum harmonic oscillator Hamiltonian defined on sufficiently nice functions in $L^2(\R,\mathrm{d}x)$ to be the operator $\mathcal{H}_{\text{HO}}f(x) = -\frac{1}{2}\frac{\mathrm{d}^2}{\mathrm{d}x^2}f(x) + \frac{1}{2}x^2f(x)$. The quantum harmonic oscillator Hamiltonian can be formally factored as $\mathcal{H}_{\text{HO}} = \frac{1}{2}\big(-\d{\mathrm{d}}{x}+x\big)\big(\d{\mathrm{d}}{x}+x\big)f(x) + \frac{1}{2}f(x)$. A standard exercise in quantum mechanics shows that the Hermite-Gauss functions, $h_l$, given by $h_l(x) = \frac{(-1)^l}{\sqrt{2^l l!\sqrt{\pi}}} e^{\frac{x^2}{2}} \frac{\mathrm{d}^l}{\mathrm{d}x^l} e^{-x^2}$, are the orthonormal $L^2(\R,\mathrm{d}x)$ eigenfunctions of the quantum harmonic oscillator Hamiltonian.

For $f\in L^2(\R,\mathrm{d}x)\cap AC(\R\,\mathrm{d}x)$ such that $xf\in L^2(\R,\mathrm{d}x)$, and $f' \in L^2(\R,\mathrm{d}x)$, define the operator $a$ by $af(x) = \frac{1}{\sqrt{2}}\big(f'(x)+xf(x)\big)$. Restricted to a similar set of functions, the adjoint $a^*$ has the form $a^*f(x) = \frac{1}{\sqrt{2}}\big(-f'(x)+xf(x)\big)$ so that $\mathcal{H}_{\text{HO}} = a^*a + \frac{1}{2}$ on sufficiently nice functions. The operators $a$, $a^*$, and $a^*a$ satisfy the following commutation relations
\begin{equation} \label{eq:oscillator_algebra}
[a,a^*] = 1, \qquad [a^*a,a] = -a, \qquad [a^*a,a^*] = a^*.
\end{equation}

A simple inspection shows that the $\operatorname{span}\{1,a,a^*, a^*a\}$ is closed under commutators and thus forms the basis of a Lie algebra. This Lie algebra is sometimes referred to as the oscillator algebra \cite{streater}. It has as a Lie subalgebra the Heisenberg-Weyl algebra which is generated by the commutation relation $[a,a^*] = 1$---also called the canonical commutation relation---in \eqref{eq:oscillator_algebra}.

The quantum harmonic oscillator Hamiltonian can be abstractly viewed as a Hamiltonian of the form $a^*a$ where $a$, $a^*$, and $a^*a$ satisfy the above commutation relations. Therefore one may view $\mathcal{H}_{\text{HO}}$ and its raising and lowering operators as a concrete representation of the abstract quantum harmonic oscillator given by \eqref{eq:oscillator_algebra} on $L^2(\R,\mathrm{d}x)$. Another representation is in terms of number states: given a separable Hilbert space $\H$ and an orthonormal basis $(e_n)_{n=0}^{\infty}$, define $a e_n = \sqrt{n} e_n$ (and extend $a$ linearly), then $a^* e_n = \sqrt{n+1}e_n$ and $a^*a e_n = n e_n$. Simple calculations show that these operators also satisfy the commutation relations in \eqref{eq:oscillator_algebra} and thus form another representation of the quantum harmonic oscillator. The number states representation also makes it easy to discuss the proper domains of definition of $a$ and $a^*$.

Yet another representation exists on the Segal-Bargmann space \cite{bargmann1, segal, bargmann2}, denoted $\F$---a Hilbert space of entire functions on $\C$ with measure $\mathrm{d}\rho(z,\bar{z}) = \frac{1}{\pi} e^{-z\bar{z}}\,\mathrm{d}A(z)$ where $\mathrm{d}A(z)$ is the usual Lebesgue measure on $\C$ given by $\mathrm{d}x\wedge \mathrm{d}y$ or $-\frac{1}{4i} \mathrm{d}z\wedge \mathrm{d}\bar{z}$. The functions $(e_n)_{n=0}^{\infty}$ given by $e_n(z) = \frac{1}{\sqrt{n!}} z^n$ form an orthogonal basis for the Segal-Bargmann space. On the Segal-Bargmann space, the complex derivative operator $\a f(z) = \d{\partial}{z}f(z)$ has adjoint $\a^*f(z) = zf(z)$. It is easy to see that the operators $\a$, $\a^*$, and $\a^*\a$ on the Segal-Bargmann space also satisfy the conditions of the oscillator algebra in \eqref{eq:oscillator_algebra}, allowing for the application of complex analytic methods to the study of the quantum harmonic oscillator.

The Segal-Bargmann transform is a Lie algebra and Hilbert space isomorphism between the $L^2(\R,\mathrm{d}x)$ representation of the oscillator algebra and the $\F$ representation of the oscillator algebra. The Segal-Bargmann transform, denoted $\mathcal{B}$, is a unitary map from $L^2(\R,\mathrm{d}x)$ to $\F$ given by
\begin{equation}
\B f(z) = \frac{1}{\pi^{\frac{1}{4}}} \int_{\R} e^{-\frac{z^2}{2}}e^{\sqrt{2}zx}e^{-\frac{x^2}{2}} f(x)\,\mathrm{d}x
\end{equation}

\noindent with inverse given by
\begin{equation}
\B^{-1} f(x) = \frac{1}{\pi^{\frac{1}{4}}} \int_{\C} e^{-\frac{\bar{z}^2}{2} + \sqrt{2} \bar{z}x - \frac{x^2}{2}} f(z)\,\mathrm{d}\rho(z,\bar{z})
\end{equation}

\noindent The Hilbert space and Lie algebra isomorphism properties that define the Segal-Bargmann transform can be summarized (for sufficiently nice $f\in L^2(\R,\mathrm{d}x)$) by:
\begin{align}
\frac{1}{\pi}\int_{\C} |\mathcal{B}f(z)|^2 e^{-z\bar{z}}\,\mathrm{d}A(z) &= \int_{-\infty}^{\infty} |f(x)|^2\,\mathrm{d}x \\
\mathcal{B}(af)(z) &= \a \mathcal{B}f(z) \\
\mathcal{B}(a^* f)(z) &= \a^*\mathcal{B}f(z) \\
\mathcal{B}(a^*a f)(z) &= \a^*\a\mathcal{B}f(z)
\end{align}

\noindent The isomorphisms can be summarized via the following commutative diagrams.

\begin{center}
 \begin{figure}[!ht]
  \begin{tikzpicture}
   \node (a) at (0,0) {$L^2(\R,\mathrm{d}x)$};
   \node (b) at (0,2) {$\F$};
   \node (c) at (4,0) {$L^2(\R,\mathrm{d}x)$};
   \node (d) at (4,2) {$\F$};
   
   \node (e) at (8,0) {$L^2(\R,\mathrm{d}x)$};
   \node (f) at (8,2) {$\F$};
   \node (g) at (12,0) {$L^2(\R,\mathrm{d}x)$};
   \node (h) at (12,2) {$\F$};
   
   \draw[->] (a) to node [left]  {$\B$}                                                   (b);
   \draw[->] (a) to node [above] {$a$} (c);
   \draw[->] (b) to node [above] {$\a$}       (d);
   \draw[->] (c) to node [left]  {$\B$}                                                   (d);
   
   \draw[->] (e) to node [left]  {$\B$}                                                   (f);
   \draw[<-] (e) to node [above] {$a^*$} (g);
   \draw[<-] (f) to node [above] {$\a^*$}       (h);
   \draw[->] (g) to node [left]  {$\B$}                                                   (h);
  \end{tikzpicture}
\caption{Commutative diagrams which define the usual Segal-Bargmann transform $\B$.}
\end{figure}
\end{center}

On the $L^2(\R,\mathrm{d}x)$ representation of the oscillator algebra, $a = \frac{1}{\sqrt{2}}\big(\d{\mathrm{d}}{x}+x\big)$ is a lowering operator and similarly $\a = \d{\mathrm{d}}{z}$ is a lowering operator in the Segal-Bargmann representation. Likewise, $a^* = \frac{1}{\sqrt{2}}\big(-\d{\mathrm{d}}{x}+x\big)$ is a raising operator on the $L^2(\R,\mathrm{d}x)$ representation of the oscillator algebra and $\a^* = z$ is a raising operator in the Segal-Bargmann representation. A straightforward induction shows that the Hermite-Gauss functions $h_l$ are mapped to $\frac{1}{\sqrt{l!}} z^l$ under the Segal-Bargmann transform.

In \cite{williams2}, the notion of a coupled supersymmetry (coupled SUSY) was introduced as a generalization of the quantum harmonic oscillator with a built-in supersymmetric nature, informed by supersymmetric quantum mechanics (SUSY QM) \cite{cooper, david}. Whereas SUSY QM only has a Lie superalgebra structure that is not strong enough to fully identify the eigenstructure of the associated quantum systems, coupled SUSYs have associated to them additional $\mathfrak{su}(1,1)$ Lie algebra structures beyond the Lie superalgebra structures which allow for the exact solvability of the eigenstructure.

Recently, there has been a wide range of Segal-Bargmann-like analysis beyond $\R^n$ to Lie groups via heat kernel and geometric quantization techniques \cite{hall1, hall2, stenzel, hilgert1, hilgert2, driver, chan, kirwin}. There have also been Segal-Bargmann transform applications to the study of superalgebras and supersymmetric quantum systems \cite{thienel, barbier1, barbier2}, as well as analogues developed in the quaternionic and Clifford algebra settings \cite{alpay1, cnudde, demartino, diki, eaknipitsari, mourao}, and $q$-deformations and related generalizations \cite{leeuwen, rosenfeld, alpay2}.

In this paper, we will present new generalizations of the usual Segal-Bargmann space and the Segal-Bargmann transform that have a supersymmetric flavor that arise from a specific class of coupled SUSYs, leveraging the $\mathfrak{su}(1,1)$ structure that come with coupled SUSYs.

The remainder of the paper is organized as follows. In \Cref{sec:2}, we will briefly discuss the notion of a coupled supersymmetry, particularly focusing on a family of coupled supersymmetries on $\R$ first introduced in \cite{williams2} of which the quantum harmonic oscillator is a special case. In \Cref{sec:3}, we will develop generalizations of the Segal-Bargmann space on which holomorphic representations of the coupled SUSYs live and remark on the reproducing kernel Hilbert space nature thereof. In \Cref{sec:4}, we will establish generalizations of the Segal-Bargmann transform that act as Lie algebra and Hilbert space isomorphisms between the two representations and discuss the relationships between the generalized Segal-Bargmann spaces. In \Cref{sec:5}, we will discuss the relationship of the coupled SUSY Segal-Bargmann transforms to short-time transforms and coherent states, particularly the work developed in \cite{barut} on $SU(1,1)$ coherent states.

\section{A Class of Coupled Supersymmetries} \label{sec:2}

\subsection{Coupled Supersymmetry}

For the sake of self-containment, we will restate some of the basics of coupled supersymmetry, including basic results without proof.

\begin{definition}
Let $\H_1$ and $\H_2$ be Hilbert spaces, $a:\H_1 \to \H_2$, and $b:\H_2\to\H_1$ be closed, densely-defined operators such that the domains are such that the proceeding identities are well-defined. We say that the ordered quadruplet $\{a,b,\gamma,\delta\}$ where $\gamma,\delta\in\R$ form a \emph{coupled supersymmetry} (\emph{coupled SUSY}) if they satisfy the following relations
\begin{align}
a^*a &= bb^* + \gamma 1_{\H_1} \label{eq:coupledsusy1}\\
aa^* &= b^*b + \delta 1_{\H_2}. \label{eq:coupledsusy2}
\end{align}

\noindent In what follows, we suppress the notation $1_{\H}$ for brevity as it is implied. The operators $a^*a$, $aa^*$, $b^*b$, and $bb^*$ are called Hamiltonians. A coupled SUSY is \emph{unbroken} if $\ker(a)$ and $\ker(b)$ are both nontrivial, \emph{partially unbroken} if exactly one of $\ker(a)$ and $\ker(b)$ is nontrivial, and \emph{broken} if $\ker(a)$ and $\ker(b)$ are both trivial.
\end{definition}

\begin{remark}
Note that taking $b = a$ and $-\gamma = 1 = \delta$ in a coupled SUSY gives the Heisenberg-Weyl commutation relations so that the abstract quantum harmonic oscillator is indeed a special case of a coupled SUSY. However, coupled SUSYs extend beyond the normal harmonic oscillator as will be seen in \Cref{def:2}.
\end{remark}

In many cases, the coupled SUSY operators ($a$, $b$, $a^*$, $b^*$) map a Hilbert space back to itself, but distinguishing the domain and codomain and allowing for the two to be the same makes the diagram chasing in what follows simpler; moreover, the nature of the operators $a$ and $b$ may necessitate defining them on separate densely-defined subspaces of $L^2$ spaces, so distinguishing their domains serves a further purpose. In traditional SUSY terms, $\H_1$ represents the Hilbert space for the first (bosonic) sector, whereas $\H_2$ represents the Hilbert space for the second (fermionic) sector \cite{cooper}.

Coupled SUSYs have an $\mathfrak{su}(1,1)$ Lie algebra structure generated by the raising and lowering operators $a^*b^*$ (or $b^*a^*$) and $ba$ (or $ab$) and the coupled SUSY Hamiltonians $a^*a$ and $bb^*$ (or $aa^*$ and $b^*b$):
\begin{align}
[a^*a, a^*b^*] &= (\delta - \gamma) a^*b^* \\
[a^*a, ba] &= -(\delta - \gamma) ba \\
[a^*b^*, ba] &= -2(\delta-\gamma) \bigg(a^*a - \frac{\gamma}{2}\bigg)
\end{align}

\noindent The Lie structure determines the eigenstructure of coupled SUSYs. In an unbroken coupled SUSY the eigenvalue structure is uniquely determined as $\ker(a)$ and $\ker(b)$ being nontrivial force zero points in the spectra of $a^*a$ and $b^*b$. In (partially) broken coupled SUSYs, the the spectra will exhibit an overall shift. In an unbroken coupled SUSY, the eigenvalues of $a^*a$ are given by $m(\delta-\gamma)$ and $m(\delta-\gamma) + \delta$ where $m\in\N_0$.

\begin{remark}
This ladder structure resembles that of the usual quantum harmonic oscillator with the squares of the creation and annihilation operators which also enjoy an $\mathfrak{su}(1,1)$ Lie algebra structure:
\begin{equation} \label{eq:square_harmonic}
\bigg[a^*a + \frac{1}{2}, \big(a^*\big)^2\bigg] = 2\big(a^*\big)^2, \qquad \bigg[a^*a + \frac{1}{2}, a^2\bigg] = -2a^2, \qquad \bigg[\big(a^*\big)^2, a^2\bigg] = -4\bigg(a^*a + \frac{1}{2}\bigg).
\end{equation}

\noindent This does not, however, encapsulate all coupled SUSYs as seen in \Cref{def:2}. Another connection exists between coupled SUSYs and the harmonic oscillator by way of the ``spinification'' of a two-particle quantum harmonic oscillator system as discussed in \cite{williams2}.
\end{remark}

\subsection{Concrete Realizations of Coupled SUSYs}

In this paper, we will be concerned with a family of coupled SUSYs indexed by $n\in\N$ where $\gamma = -1$ and $\delta = 2n-1$, i.e.
\begin{align}
a^* a &= b b^* - 1 \label{eq:coupledsusy_real1} \\
a a^* &= b^* b + 2n-1. \label{eq:coupledsusy2_real2}
\end{align}

\noindent We will suppress the explicit dependence on $n$ of the operators $a$, $a^*$, $b$, and $b^*$ throughout for ease of notation.

\begin{definition}\label{def:2}
Define the operators $a$ and $b$ on $L^2(\R,\mathrm{d}x)$ by
\begin{equation}
a = \frac{1}{\sqrt{2}}\bigg(\frac{1}{x^{n-1}}\d{\mathrm{d}}{x} + x^n\bigg), \qquad b = \frac{1}{\sqrt{2}}\bigg(\d{\mathrm{d}}{x}\frac{1}{x^{n-1}} + x^n\bigg)
\end{equation}
\end{definition}

These operators were introduced in \cite{williams2}. $a$ can be formally defined on those $f\in L^2(\R,\mathrm{d}x)$ such that $x^n f\in L^2(\R,\mathrm{d}x)$ and $\frac{1}{x^{n-1}} f'\in L^2(\R,\mathrm{d}x)$, similar for $b$. Examples of such functions include the compactly supported smooth functions supported away from zero. A straightforward computation shows that on sufficiently nice functions, e.g. the compactly supported smooth functions supported away from zero, the adjoints $a^*$ and $b^*$ can be represented by
\begin{equation}
a^* = \frac{1}{\sqrt{2}}\bigg(-\d{\mathrm{d}}{x}\frac{1}{x^{n-1}} + x^n\bigg), \qquad b^* = \frac{1}{\sqrt{2}}\bigg(-\frac{1}{x^{n-1}}\d{\mathrm{d}}{x} + x^n\bigg).
\end{equation}

\noindent Note the functional similarity between $a^*$ and $b$ and likewise $a$ and $b^*$. This plays a crucial role in the analysis that follows.

\begin{remark}
Taking $n=1$ in the above reveals the quantum harmonic oscillator ladder operators so that the $n=1$ case in what follows will reduce to the usual Segal-Bargmann transform and Segal-Bargmann space. We will remark on this special case throughout as we build up the coupled SUSY Segal-Bargmann spaces and transforms.
\end{remark}

In a slight abuse of notation in light of the definition of a coupled SUSY, we will reserve $\H_1$ and $\H_2$ henceforth to be $L^2(\R,\mathrm{d}x)$ on which the above representation lives. $\H_1$ represents the Hilbert space corresponding to the eigenfunctions of $a^*a$ (or equivalently, $bb^*$), whereas $\H_2$ represents the Hilbert space corresponding to the eigenfunctions of $aa^*$ (or equivalently, $b^*b$).

\begin{definition}
Define the normalized functions $\psi_0 \in \H_1$ and $\widetilde{\psi}_0 \in \H_2$ by
\begin{equation}
\psi_0(x) = \frac{n^{\frac{1}{2}-\frac{1}{4n}}}{\sqrt{\Gamma\big(\frac{1}{2n}\big)}} e^{-\frac{x^{2n}}{2n}}, \qquad
\widetilde{\psi}_0(x) = \frac{n^{\frac{1}{4n}}}{\sqrt{\Gamma\big(1-\frac{1}{2n}\big)}} x^{n-1}e^{-\frac{x^{2n}}{2n}}
\end{equation}

\noindent Further, define the higher (normalized) eigenfunctions $\psi_l$ of $a^*a$ and $\widetilde{\psi}_l$ of $aa^*$ by
\begin{align}
\psi_{2l} &= \frac{(a^*b^*)^l \psi_0}{\|(a^*b^*)^l \psi_0\|}, \\
\psi_{2l+1} &= \frac{(a^*b^*)^l a^*\widetilde{\psi}_0}{\|(a^*b^*)^l a^*\widetilde{\psi}_0\|}, \\
\widetilde{\psi}_{2l} &= \frac{(b^*a^*)^l\widetilde{\psi}_0}{\|(b^*a^*)^l\widetilde{\psi}_0\|}, \\
\widetilde{\psi}_{2l+1} &= \frac{(b^*a^*)^l b^*\psi_0}{\|(b^*a^*)^l b^*\psi_0\|}.
\end{align}
\end{definition}

A simple computation shows that
\begin{equation}
a \psi_0 = 0, \qquad b \widetilde{\psi}_0 = 0,
\end{equation}

\noindent so that $a^*a \psi_0 = 0$ and $b^*b\widetilde{\psi}_0 = 0$ so that these are indeed eigenfunctions and similarly, $\psi_l$ and $\widetilde{\psi}_l$ are eigenfunctions of their corresponding Hamiltonians, seen by way of the $\mathfrak{su}(1,1)$ Lie algebra structure.

\begin{proposition}
The eigenfunctions of $a^*a$ generated in this way have Rodrigues formulae resembling that for the Hermite-Gauss functions given by
\begin{align}
 (a^*b^*)^l e^{-\frac{x^{2n}}{2n}} &= \frac{1}{2^l} e^{\frac{x^{2n}}{2n}}\bigg(\d{\mathrm{d}}{x}\frac{1}{x^{2n-2}}\d{\mathrm{d}}{x}\bigg)^l e^{-\frac{x^{2n}}{n}} \label{eq:psi_l1} \\
(a^*b^*)^l \bigg(2x^{2n-1}e^{-\frac{x^{2n}}{2n}}\bigg) &= \frac{1}{2^l} e^{\frac{x^{2n}}{2n}}\bigg(\d{\mathrm{d}}{x}\frac{1}{x^{2n-2}}\d{\mathrm{d}}{x}\bigg)^l\bigg(2x^{2n-1} e^{-\frac{x^{2n}}{n}}\bigg). \label{eq:psi_l2}
\end{align}

\noindent Similar formulae exist for the eigenfunctions of $aa^*$. The proof is a straightforward induction.
\end{proposition}

\begin{remark}
When $n=1$, $\psi_0(x) = \pi^{-\frac{1}{4}} e^{-\frac{x^2}{2}} = \widetilde{\psi}_0(x)$, giving the usual quantum harmonic oscillator ground state, and successive application of the raising operators gives the quantum harmonic oscillator excited states, i.e. the Hermite-Gauss functions.
\end{remark}

\begin{proposition}
The eigenfunctions $\psi_l$ and $\widetilde{\psi}_l$ form orthonormal bases for $\H_1$ and $\H_2$, respectively.
\end{proposition}

A proof of related facts can be found in \cite{williams1} and is based on a proof by Akhiezer for the Fourier-Bessel (Hankel) transform \cite{akhiezer}. Therein, functions directly related to the functions $\psi_l$ were shown to give an orthonormal basis for $L^2(\R,\mathrm{d}x)$. The proof can be adapted readily to $\psi_l$ and $\widetilde{\psi}_l$ to show completeness.

The coupled SUSY ladder structure is summarized in the diagrams below. Note that the notation differs slightly from \cite{williams2} for ease of discussion.

\begin{figure}[!ht]
\centering
\begin{tikzpicture}
\node (a) at (1.5,0) {};
\node (b) at (3.5,0) {};
\draw (a) to node [above] {$\H_1$} (b);

\node (c) at (7.5,0) {};
\node (d) at (9.5,0) {};
\draw (c) to node [above] {$\H_2$} (d);

\draw (1,-5)--(4,-5);
\draw (1,-4)--(4,-4);
\draw (1,-3.5)--(4,-3.5);
\draw (1,-2.5)--(4,-2.5);
\draw (1,-2)--(4,-2);
\draw (1,-1)--(4,-1);
\draw (1,-0.5)--(4,-0.5);

\draw (7,-4)--(10,-4);
\draw (7,-3.5)--(10,-3.5);
\draw (7,-2.5)--(10,-2.5);
\draw (7,-2)--(10,-2);
\draw (7,-1)--(10,-1);
\draw (7,-0.5)--(10,-0.5);

\node (1a) at (4,-4) {};
\node (1b) at (7,-4) {};

\node (2a) at (4,-2.5) {};
\node (2b) at (7,-1) {};

\node at (0.5,-5) {$\psi_0$};
\node at (0.5,-4) {$\psi_1$};
\node at (0.5,-3.5) {$\psi_2$};
\node at (0.5,-2.5) {$\psi_3$};
\node at (0.5,-2) {$\psi_4$};
\node at (0.5,-1) {$\psi_5$};
\node at (0.5,-0.5) {$\psi_6$};

\node at (10.5,-4) {$\wt\psi_0$};
\node at (10.5,-3.5) {$\wt\psi_1$};
\node at (10.5,-2.5) {$\wt\psi_2$};
\node at (10.5,-2) {$\wt\psi_3$};
\node at (10.5,-1) {$\wt\psi_4$};
\node at (10.5,-0.5) {$\wt\psi_5$};

\draw[->] (1a) to [bend right=30] node[below] {$a$} (1b);
\draw[->] (1b) to [bend right=30] node[below] {$a^*$} (1a);

\draw[->] (2a) to [bend right=30] node[below] {$b^*$} (2b);
\draw[->] (2b) to [bend right=30] node[below] {$b$} (2a);
\end{tikzpicture}
\caption{The actions of $a$, $b$, $a^*$, and $b^*$ in a coupled SUSY.}

\hfill

\begin{tikzpicture}
\node (a) at (1.5,0) {};
\node (b) at (3.5,0) {};
\draw (a) to node [above] {$\H_1$} (b);

\node (c) at (7.5,0) {};
\node (d) at (9.5,0) {};
\draw (c) to node [above] {$\H_2$} (d);

\draw (1,-2)--(4,-2);
\draw (1,-1)--(4,-1);
\draw (1,-0.5)--(4,-0.5);

\draw (7,-1)--(10,-1);
\draw (7,-0.5)--(10,-0.5);

\node (1a) at (4,-0.5) {};
\node (1b) at (7,-0.5) {};

\node (2a) at (4,-2) {};
\node (2b) at (7,-0.5) {};

\node (3a) at (0.1,-2) {};
\node (3b) at (0.1,-0.5) {};

\node at (0.5,-2) {$\psi_0$};
\node at (0.5,-1) {$\psi_1$};
\node at (0.5,-0.5) {$\psi_2$};

\node at (10.5,-1) {$\wt\psi_0$};
\node at (10.5,-0.5) {$\wt\psi_1$};

\draw[->] (1b) to [bend right=30] node[below] {$a^*$} (1a);

\draw[->] (2a) to [bend right=30] node[below] {$b^*$} (2b);

\draw[->] (3a) to [bend left=30] node[left] {$a^*b^*$} (3b);
\end{tikzpicture}
\caption{The raising operator structure for the first sector in a coupled SUSY.}

\hfill

\begin{tikzpicture}
\node (a) at (1.5,0) {};
\node (b) at (3.5,0) {};
\draw (a) to node [above] {$\H_1$} (b);

\node (c) at (7.5,0) {};
\node (d) at (9.5,0) {};
\draw (c) to node [above] {$\H_2$} (d);

\draw (1,-2)--(4,-2);
\draw (1,-1)--(4,-1);
\draw (1,-0.5)--(4,-0.5);

\draw (7,-1)--(10,-1);
\draw (7,-0.5)--(10,-0.5);

\node (1a) at (4,-0.5) {};
\node (1b) at (7,-0.5) {};

\node (2a) at (4,-2) {};
\node (2b) at (7,-0.5) {};

\node (3a) at (0.1,-2) {};
\node (3b) at (0.1,-0.5) {};

\node at (0.5,-2) {$\psi_0$};
\node at (0.5,-1) {$\psi_1$};
\node at (0.5,-0.5) {$\psi_2$};

\node at (10.5,-1) {$\wt\psi_0$};
\node at (10.5,-0.5) {$\wt\psi_1$};

\draw[<-] (1b) to [bend right=30] node[below] {$a$} (1a);

\draw[<-] (2a) to [bend right=30] node[below] {$b$} (2b);

\draw[<-] (3a) to [bend left=30] node[left] {$ba$} (3b);
\end{tikzpicture}
\caption{The lowering operator structure for the first sector in a coupled SUSY.}
\end{figure}

\section{Constructing the Coupled SUSY Segal-Bargmann Spaces} \label{sec:3}

The following relations hold for the complex variable $z\in\C$ and holomorphic derivative $\d{\mathrm{d}}{z}$ as operators:
\begin{align}
\big(z^n\big)\bigg(\frac{1}{z^{n-1}}\d{\mathrm{d}}{z}\bigg) &= \bigg(\d{\mathrm{d}}{z} \frac{1}{z^{n-1}}\bigg)\big(z^n\big) - 1, \\
\bigg(\frac{1}{z^{n-1}}\d{\mathrm{d}}{z}\bigg)\big(z^n\big) &= \big(z^n\big)\bigg(\d{\mathrm{d}}{z}\frac{1}{z^{n-1}}\bigg) + 2n-1,
\end{align}

\noindent mimicking the coupled SUSY relations in \eqref{eq:coupledsusy1} and \eqref{eq:coupledsusy2}.

In order for these operators and their corresponding products to be well-defined on entire functions, we must restrict the domains for $\frac{1}{z^{n-1}}\d{\mathrm{d}}{z}$ and $\d{\mathrm{d}}{z}\frac{1}{z^{n-1}}$. Define $\dom\big(\frac{1}{z^{n-1}}\d{\mathrm{d}}{z}\big)$ to be those entire functions spanned by $\{1, z^{2n-1}, z^{2n}, z^{4n-1}, z^{4n}, \ldots\}$ and $\dom\big(\d{\mathrm{d}}{z}\frac{1}{z^{n-1}}\big)$ to be those entire functions spanned by \sloppy $\{z^{n-1}, z^n, z^{3n-1}, z^{3n}, \ldots\}$. It is easy to see that these are the maximal subspaces for which the coupled SUSY domain and range conditions hold.

\begin{definition}
To this end, we define the following not-yet-closed spaces of entire functions:
\begin{align}
\Os_1(\C) &= \operatorname{span}\{1, z^{2n-1}, z^{2n}, z^{4n-1}, z^{4n}, \ldots\} \label{eq:coupled_susy_complex_1} \\
\Os_2(\C) &= \operatorname{span}\{z^{n-1}, z^{n}, z^{3n-1}, z^{3n}, \ldots\} \label{eq:coupled_susy_complex_2}
\end{align}
\end{definition}

\begin{remark}
Such choices for these spaces further distinguish this from the square operator representation given in \eqref{eq:square_harmonic} as these would lead to the usual Segal-Bargmann space---or perhaps the even and odd subspaces thereof, depending on interpretation.
\end{remark}

\begin{definition}
Let $\rho_1$ and $\rho_2$ be not-yet-defined weights on $\Os_1(\C)$ and $\Os_2(\C)$, respectively, defining inner products $\langle \cdot, \cdot\rangle_1$ and $\langle \cdot, \cdot\rangle_2$, respectively, by
\begin{align}
\langle f_1, g_1\rangle_1 &= \int_{\C} f_1(z) \overline{g_1(z)} \rho_1(z,\bar{z})\,\mathrm{d}A(z) \\
\langle f_2, g_2\rangle_2 &= \int_{\C} f_2(z) \overline{g_2(z)} \rho_2(z,\bar{z})\,\mathrm{d}A(z).
\end{align}

\noindent Let $\|\cdot\|_1$ denote the norm resulting from $\langle \cdot, \cdot\rangle_1$ and Let $\|\cdot\|_2$ denote the norm resulting from $\langle \cdot, \cdot\rangle_2$.
\end{definition}

Much like in \cite{bargmann1}, $\rho_1$ and $\rho_2$ will be defined so that the adjoints of $\frac{1}{z^{n-1}}\d{\mathrm{d}}{z}$ and $\d{\mathrm{d}}{z}\frac{1}{z^{n-1}}$ are both $z^n$, agreeing with the desired coupled SUSY structure.

If we wish to realize $\frac{1}{z^{n-1}}\d{\mathrm{d}}{z}$ and $\d{\mathrm{d}}{z}\frac{1}{z^{n-1}}$ as a representation of the coupled SUSY given in \eqref{eq:coupled_susy_complex_1} and \eqref{eq:coupled_susy_complex_2}, then the following relations must hold for $f_1,g_1 \in \Os_1(\C)$ and $f_2, g_2\in \Os_2(\C)$:
\begin{align}
\int_{\C} \frac{1}{z^{n-1}}\d{\mathrm{d}}{z}f_1(z) \overline{g_2(z)} \rho_2(z,\bar{z})\,\mathrm{d}A(z) &= \int_{\C} f_1(z) \overline{z^n g_2(z)} \rho_1(z,\bar{z})\,\mathrm{d}A(z) \label{eq:rho_1_condition}\\
\int_{\C} \d{\mathrm{d}}{z}\frac{1}{z^{n-1}}f_2(z) \overline{g_1(z)} \rho_1(z,\bar{z})\,\mathrm{d}A(z) &= \int_{\C} f_2(z) \overline{z^n g_1(z)} \rho_2(z,\bar{z})\,\mathrm{d}A(z) \label{eq:rho_2_condition}
\end{align}

\noindent Integrating by parts in \eqref{eq:rho_1_condition} and \eqref{eq:rho_2_condition}, making use of entirety of $g_1$ and $g_2$, and assuming---for now---that the boundary terms go to zero gives
\begin{align}
\int_{\C} f_1(z) \overline{g_2(z)} \overline{z}^n\rho_1(z,\bar{z})\,\mathrm{d}A(z) &= \int_{\C} f_1(z) \overline{g_2(z)} \bigg(-\d{\partial}{z} \frac{1}{z^{n-1}} \rho_2(z,\bar{z})\bigg)\,\mathrm{d}A(z), \\
\int_{\C} f_2(z) \overline{g_1(z)} \overline{z}^n\rho_2(z,\bar{z})\,\mathrm{d}A(z) &= \int_{\C} f_2(z) \overline{g_1(z)} \bigg(-\frac{1}{z^{n-1}}\d{\partial}{z} \rho_1(z,\bar{z})\bigg)\,\mathrm{d}A(z).
\end{align}

\noindent It is easy to see that the following relationships for $\rho_1$ and $\rho_2$ guarantee that the above hold:
\begin{align}
-\dfrac{\partial}{\partial z} \dfrac{1}{z^{n-1}} \rho_2(z,\bar{z}) &= \bar{z}^n \rho_1(z,\bar{z})\\ 
-\dfrac{1}{z^{n-1}} \dfrac{\partial}{\partial z} \rho_1(z,\bar{z}) &= \bar{z}^n \rho_2(z,\bar{z}).
\end{align}

\noindent Decoupling these equations leads to the following partial differential equations for $\rho_1$ and $\rho_2$:
\begin{align}
\d{\partial}{z} \frac{1}{z^{2n-2}} \d{\partial}{z} \rho_1(z,\bar{z}) &= \bar{z}^{2n} \rho_1(z,\bar{z}) \\
\frac{1}{z^{n-1}}\frac{\partial^2}{\partial z^2} \frac{1}{z^{n-1}} \rho_2(z,\bar{z}) &= \bar{z}^{2n} \rho_2(z,\bar{z}).
\end{align}

\begin{definition}
Define $\rho_1:\C^2\to\R$ and $\rho_2:\C^2\to\R$ to be
\begin{align}
 \rho_1(z,\overline{z}) &= \frac{2}{(2n)^{\frac{1}{2n}} \pi \Gamma\big(\frac{1}{2n}\big)} (z\bar{z})^{n-\frac{1}{2}} K_{1-\frac{1}{2n}}\bigg(\frac{(z\bar{z})^n}{n}\bigg) \\
 \rho_2(z,\overline{z}) &= \frac{2}{(2n)^{\frac{1}{2n}} \pi \Gamma\big(\frac{1}{2n}\big)} (z\bar{z})^{n-\frac{1}{2}} K_{\frac{1}{2n}}\bigg(\frac{(z\bar{z})^n}{n}\bigg),
\end{align}

\noindent where $K_{\nu}$ is the modified Bessel function of the second kind \cite[Eq.~10.27.4]{nist} given by
\begin{equation}
K_{\nu}(z) = \frac{\pi}{2\sin(\nu\pi)}(I_{-\nu}(z) - I_{\nu}(z)),
\end{equation}

\noindent and $I_{\nu}$ is the modified Bessel function of the first kind \cite[Eq.~10.25.2]{nist} given by
\begin{equation}
I_{\nu}(z) = \bigg(\frac{z}{2}\bigg)^{\nu} \sum_{l=0}^{\infty} \frac{1}{\Gamma(l+\nu+1)l!} \bigg(\frac{z}{2}\bigg)^{2l}.
\end{equation}
\end{definition}

From these representations in terms of the modified Bessel function $K_{\nu}$, we can see that $\rho_1$ and $\rho_2$ solve the above partial differential equations. Note that both $\rho_1$ and $\rho_2$ share a normalization factor of $\frac{2}{(2n)^{\frac{1}{2n}} \pi \Gamma\big(\frac{1}{2n}\big)}$. This is due to the coupled nature of the two weight functions, i.e. they cannot be separately normalized to $1$.

In order for the inner products $\langle \cdot,\cdot\rangle_1$ and $\langle\cdot,\cdot\rangle_2$ to be positive, $\rho_1$ and $\rho_2$ should be positive functions. The modified Bessel function $K_{\nu}$ has the following integral representation \cite[Eq.~10.32.9]{nist}:
\begin{equation}
K_{\nu}(x) = \int_0^{\infty} e^{-x\cosh(t)} \cosh(\nu t)\,\mathrm{d}t, \qquad\qquad |\operatorname{ph}(x)| < \frac{\pi}{2}.
\end{equation}

\noindent For positive $x$, the integrand is real and positive. Taking $x = z\bar{z}$, where $z\in\C$, we see that $K_{\nu}\big(\frac{(z\bar{z})^n}{n}\big)$ is positive and thus so are $\rho_1$ and $\rho_2$.

\begin{remark}
When $n=1$, $\rho_1$ and $\rho_2$ are identical and reduce to a single weight $\rho$ which is given by
\begin{equation}
\rho(z,\bar{z}) = \frac{1}{\pi}\sqrt{\frac{2}{\pi}} |z| K_{\frac{1}{2}}\big(|z|^2\big).
\end{equation}

\noindent Noting that for $x > 0$, $K_{\frac{1}{2}}(x) = \sqrt{\frac{\pi}{2x}} e^{-x}$, \cite[Eq.~10.39.2]{nist} we obtain
\begin{equation}
\rho(z,\bar{z}) = \frac{1}{\pi} e^{-z\bar{z}},
\end{equation}

\noindent matching the usual Segal-Bargmann space weight function as expected.
\end{remark}

The modified Bessel function $K_{\nu}$ has a nice asymptotic expansion as $x\to\infty$ \cite[Eq.~10.40.2]{nist}:
\begin{equation}
 K_{\nu}(x) \sim \sqrt{\frac{\pi}{2x}} e^{-x}\bigg(1 + \frac{4\nu^2-1}{8x} + O(x^{-2})\bigg).
\end{equation}

\noindent Thus, asymptotically, $\rho_1$ and $\rho_2$ have the following asymptotic expansions:
\begin{align}
 \rho_1(z,\overline{z}) &\sim \frac{1}{(2n)^{\frac{1}{2n}} \Gamma\big(\frac{1}{2n}\big)}\sqrt{\frac{2}{\pi}} |z|^{n-1} e^{-\frac{|z|^{2n}}{n}}\bigg(1+ \alpha |z|^{-2n} + O(z^{-4n})\bigg) \\
 \rho_2(z,\overline{z}) &\sim \frac{1}{(2n)^{\frac{1}{2n}} \Gamma\big(\frac{1}{2n}\big)}\sqrt{\frac{2}{\pi}} |z|^{n-1} e^{-\frac{|z|^{2n}}{n}}\bigg(1+\alpha |z|^{-2n} + O(z^{-4n})\bigg).
\end{align}

\noindent Therefore the weights $\rho_1$ and $\rho_2$ are exponentially decaying, much as in the case of the traditional Segal-Bargmann space, which gives that the monomials $z^n$ have finite norm. Indeed for $n=1$, the asymptotic expansion is exact and is a Gaussian in $|z|$. This informs the assumption that the boundary terms in the integration by parts in \eqref{eq:rho_1_condition} and \eqref{eq:rho_2_condition} vanish.

Much as in the case of the Segal-Bargmann space, the functions $z^{2nl}$ and $z^{2nl+2n-1}$ in $\Os_1(\C)$ and $z^{2nl+n-1}$ and $z^{2nl+n}$ in $\Os_2(\C)$ are not normalized.

\begin{definition} Define the normalized functions $e_l$ and $\widetilde{e}_l$ by
\begin{align}
e_l(z) &= \begin{cases} \sqrt{\frac{\Gamma\big(\frac{1}{2n}\big)}{(2n)^{2k} \Gamma\big(k+\frac{1}{2n}\big)k!}} z^{2nk}, & l = 2k \\ \sqrt{\frac{\Gamma\big(\frac{1}{2n}\big)}{(2n)^{2k+2-\frac{1}{n}} \Gamma\big(k+2-\frac{1}{2n}\big)k!}} z^{2nk+2n-1}, & l = 2k+1 \end{cases} \\
\widetilde{e}_l(z) &= \begin{cases} \sqrt{\frac{\Gamma\big(\frac{1}{2n}\big)}{(2n)^{2k+1-\frac{1}{n}} \Gamma\big(k+1-\frac{1}{2n}\big)k!}} z^{2nk+n-1}, & l = 2k \\ \sqrt{\frac{\Gamma\big(\frac{1}{2n}\big)}{(2n)^{2k+1} \Gamma\big(k+1+\frac{1}{2n}\big)k!}} z^{2nk+n}, & l = 2k+1 \end{cases}
\end{align}

\noindent Note that $e_0 \equiv 1$, but $\widetilde{e}_0 \not\equiv 1$ in general, a sharp contrast from typical Segal-Bargmann spaces.
\end{definition}

\begin{definition}
Let $\F_1$ and $\F_2$ be the holomorphic function spaces corresponding to $\Os_1(\C)$ and $\Os_2(\C)$, respectively, i.e.
\begin{align}
\F_1 &= \bigg\{ f_1 = \sum_l c_l e_l \biggm| \sum_l |c_l|^2 < \infty\bigg\}, \\
\F_2 &= \bigg\{ f_2 = \sum_l d_l \widetilde{e}_l \biggm| \sum_l |d_l|^2 < \infty\bigg\}.
\end{align}
\end{definition}

$\F_1$ and $\F_2$ are Hilbert spaces as they are equivalent to an $\ell^2$ space. Standard arguments \cite{hallnotes} show that point evaluation is a bounded linear functional on $\F_1$ and $\F_2$ which allows norm convergence to pass to uniform convergence on compact sets. As uniform limits of holomorphic functions on compact sets are again holomorphic, we can conclude that $\F_1$ and $\F_2$ are indeed comprised of entire functions---not merely $L^2$ limits of entire functions.

With the establishment of the holomorphic function spaces $\F_1$ and $\F_2$, we can define the holomorphic realization of the coupled SUSY operators.

\begin{definition}
Define the operators $\a: \F_1 \to \F_2$ and $\b:\F_2\to\F_1$ by
\begin{align}
\a f_1(z) &= \frac{1}{z^{n-1}}\d{\mathrm{d}}{z}f_1(z) \\
\b f_2(z) &= \d{\mathrm{d}}{z}\frac{1}{z^{n-1}}f_2(z).
\end{align}
\end{definition}

Since $\a:\F_1\to\F_2$, $\a^*:\F_2\to\F_1$, and similarly since $\b:\F_2\to\F_1$, $\b^*:\F_1\to\F_2$. From previous arguments by way of the constructions of $\rho_1$ and $\rho_2$, we have that
\begin{align}
\a^* f_2(z) &= z^n f_2(z) \\
\b^* f_1(z) &= z^n f_1(z).
\end{align}

With the adjoints now established, we see that the ordered quadruplet $\{\a, \b, -1, 2n-1\}$ forms a coupled SUSY on the spaces $\F_1$ and $\F_2$.

In a further parallel with the typical Segal-Bargmann space, the holomorphic function spaces $\F_1$ and $\F_2$ are reproducing kernel Hilbert spaces \cite{paulsen}, a fact that will prove useful in the next section. We state the formal definition of a reproducing kernel Hilbert space (RKHS) given in \cite{paulsen} for self-containment.

\begin{definition}
Let $X$ be a set and $\H$ be a vector subspace of the complex-valued functions defined on $X$ with Hilbert space structure. $\H$ is called a \emph{reproducing kernel Hilbert space} if the evaluation functional is a continuous linear functional on $\H$. The function $f_x$ satisfying $\langle h, f_x\rangle = h(x)$ for all $x\in X$ and $h\in\H$ is called the \emph{reproducing kernel}.
\end{definition}

\begin{theorem} \label{thm:1}
The holomorphic function spaces $\F_1$ and $\F_2$ are reproducing kernel Hilbert spaces with the respective reproducing kernels $F_w$ and $\widetilde{F}_w$ given by
\begin{align}
F_w(z) &= \sum_{l=0}^{\infty} e_l(z)\overline{e_l(w)} \\
\widetilde{F}_w(z) &= \sum_{l=0}^{\infty} \widetilde{e}_l(z) \overline{\widetilde{e}_l(w)}
\end{align}
\end{theorem}

\begin{proof}
Elementary methods show that the series converge for all $z\in\C$ and thus guarantee the entirety of $F_w$ and $\widetilde{F}_w$ (as functions of $z$). From the definitions of $e_l$ and $\widetilde{e}_l$, it is straightforward to see that $F_w$ and $\widetilde{F}_w$ can be represented by the hypergeometric function ${}_0F_1$ \cite[Eq.~16.2.1]{nist}
\begin{align}
F_w(z) &= {}_0F_1\bigg(;\frac{1}{2n};\frac{(z\overline{w})^{2n}}{(2n)^2}\bigg) + \frac{\Gamma\big(\frac{1}{2n}\big)}{\Gamma\big(2-\frac{1}{2n}\big)} \frac{(z\overline{w})^{2n-1}}{(2n)^{2-\frac{1}{n}}} {}_0F_1\bigg(;2-\frac{1}{2n}; \frac{(z\overline{w})^{2n}}{(2n)^2}\bigg), \\
\widetilde{F}_w(z) &= \frac{\Gamma\big(\frac{1}{2n}\big)}{\Gamma\big(1-\frac{1}{2n}\big)}\frac{(z\overline{w})^{n-1}}{(2n)^{1-\frac{1}{n}}}{}_0F_1\bigg(;1-\frac{1}{2n};\frac{(z\overline{w})^{2n}}{(2n)^2}\bigg) + (z\overline{w})^n {}_0F_1\bigg(;1+\frac{1}{2n}; \frac{(z\overline{w})^{2n}}{(2n)^2}\bigg).
\end{align}

We now show that $F_w\in\F_1$ and $\widetilde{F}_w\in\F_2$ for fixed $w\in\C$. To this end we compute $\langle F_w, F_w\rangle_1$ and $\langle\widetilde{F}_w, \widetilde{F}_w\rangle_2$. Noting that $\langle e_l, e_l\rangle_1 = 1 = \langle \widetilde{e}_l, \widetilde{e}_l\rangle$ for all $l$,
\begin{align}
\langle F_w, F_w\rangle_1 &= \sum_{l=0}^{\infty} \frac{\Gamma\big(\frac{1}{2n}\big) |w|^{4nl}}{(2n)^{2l}\Gamma\big(l + \frac{1}{2n}\big) l!} + \sum_{l=0}^{\infty} \frac{\Gamma\big(\frac{1}{2n}\big)|w|^{4nl+4n-2}}{(2n)^{2l+2-\frac{1}{n}} \Gamma\big(l+2-\frac{1}{2n}\big)l!} \\
\langle \widetilde{F}_w, \widetilde{F}_w\rangle_2 &= \sum_{l=0}^{\infty} \frac{\Gamma\big(\frac{1}{2n}\big) |w|^{4nl+2n-2}}{(2n)^{2l+1-\frac{1}{n}}\Gamma\big(l + 1 - \frac{1}{2n}\big) l!} + \sum_{l=0}^{\infty} \frac{\Gamma\big(\frac{1}{2n}\big)|w|^{4nl+2n}}{(2n)^{2l+1} \Gamma\big(l+1+\frac{1}{2n}\big)l!}.
\end{align}

\noindent A straightforward analysis shows that each series converges for all $w\in\C$ and so both inner products are finite. Thus $F_w\in\F_1$ and $\widetilde{F}_w\in\F_2$ for all $w\in\C$.

That $F_w$ and $\widetilde{F}_w$ are reproducing kernels follows directly from the fact that $\langle e_l, F_w\rangle_1 = e_l(w)$ and $\langle \widetilde{e}_l, \widetilde{F}_w\rangle_2 = \widetilde{e}_l(w)$, extended linearly to the whole space, noting that limits pass through by virtue of $\|F_w\|_1 < \infty$ and $\|\widetilde{F}_w\|_2 < \infty$.
\end{proof}

\begin{remark}
When $n=1$, $F_w$ and $\widetilde{F}_w$ simplify to $e^{z\bar{w}}$, matching the reproducing kernel in the Segal-Bargmann space as expected.
\end{remark}

In the next section, we will explore the relationship between the real line realization of the coupled SUSY and the holomorphic realization and develop a generalization of the Segal-Bargmann transform.

\section{The Coupled SUSY Segal-Bargmann Transforms} \label{sec:4}

As the ordered quadruplet $\{\a, \b, -1, 2n-1\}$ satisfies the same coupled SUSY relations as the ordered quadruplet $\{a, b, -1, 2n-1\}$, it is natural to find operators---the coupled SUSY Segal-Bargmann transforms---which make this isomorphism explicit. We would like the Segal-Bargmann transforms to map $a$ to $\a$, $b$ to $\b$, $a^*$ to $\a^*$, and $b^*$ to $\b^*$. Let the coupled SUSY Segal-Bargmann transforms $\B_1:\H_1\to\F_1$ and $\B_2:\H_2\to\F_2$ to be the bounded operators defined via the following commutative diagrams.
\begin{center}
 \begin{figure}[!ht]
  \begin{tikzpicture}
   \node (a) at (0,0) {$\H_1$};
   \node (b) at (0,2) {$\F_1$};
   \node (c) at (5,0) {$\H_2$};
   \node (d) at (5,2) {$\F_2$};
   
   \node (e) at (8,0) {$\H_1$};
   \node (f) at (8,2) {$\F_1$};
   \node (g) at (13,0) {$\H_2$};
   \node (h) at (13,2) {$\F_2$};
   
   \draw[->] (a) to node [left]  {$\B_1$}                                                   (b);
   \draw[->] (a) to node [above] {$a$} (c);
   \draw[->] (b) to node [above] {$\a$}       (d);
   \draw[->] (c) to node [left]  {$\B_2$}                                                   (d);
   
   \draw[->] (e) to node [left]  {$\B_1$}                                                   (f);
   \draw[<-] (e) to node [above] {$a^*$} (g);
   \draw[<-] (f) to node [above] {$\a^*$}       (h);
   \draw[->] (g) to node [left]  {$\B_2$}                                                   (h);
  \end{tikzpicture}
\end{figure}
\begin{figure}[!ht]
  \begin{tikzpicture}
   \node (a) at (0,0) {$\H_1$};
   \node (b) at (0,2) {$\F_1$};
   \node (c) at (5,0) {$\H_2$};
   \node (d) at (5,2) {$\F_2$};
   
   \node (e) at (8,0) {$\H_1$};
   \node (f) at (8,2) {$\F_1$};
   \node (g) at (13,0) {$\H_2$};
   \node (h) at (13,2) {$\F_2$};
   
   \draw[->] (a) to node [left]  {$\B_1$}                                                      (b);
   \draw[<-] (a) to node [above] {$b$}  (c);
   \draw[<-] (b) to node [above] {$\b$}        (d);
   \draw[->] (c) to node [left]  {$\B_2$}                                                      (d);
   
   \draw[->] (e) to node [left]  {$\B_1$}                                                   (f);
   \draw[<-] (g) to node [above] {$b^*$} (e);
   \draw[<-] (h) to node [above] {$\b^*$}       (f);
   \draw[->] (g) to node [left]  {$\B_2$}                                                   (h);
  \end{tikzpicture}
 \caption{Commutative diagrams which define the coupled SUSY Segal-Bargmann transforms $\B_1$ and $\B_2$}
 \end{figure}
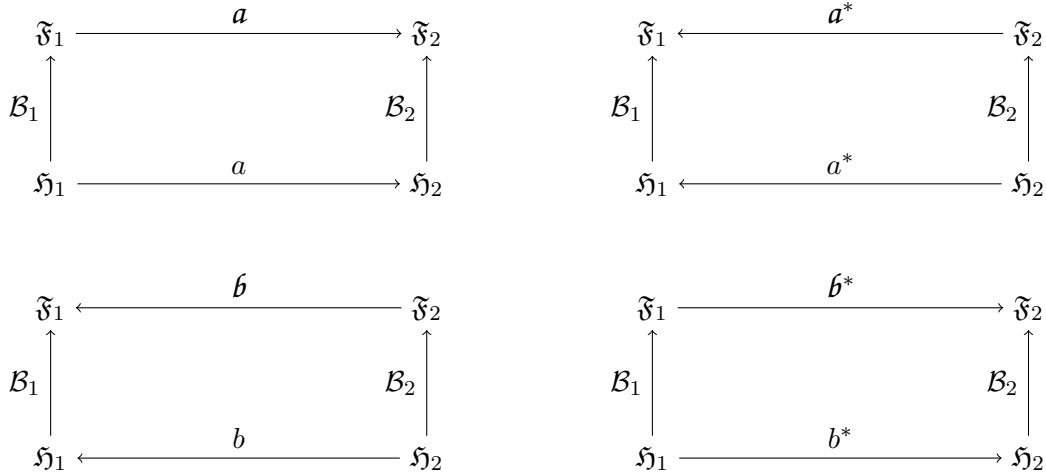
\end{center}

For $f_1\in \H_1$, $\B_1 f_1$ is a holomorphic function and is therefore well-defined pointwise, similarly for $f_2\in \H_2$ and $\B_2 f_2$. As the evaluation map $z \mapsto \B_1 f_1(z)$ is bounded by way of \Cref{thm:1}, $|\B_1 f_1(z)| \le C \||\B_1 f_1\|_1$. Furthermore since $\B_1$ is assumed to be bounded, $|\B_1 f_1(z)| \le D \|f_1\|$ and so by the Riesz representation theorem, there exists a function $A_1(z,\cdot)\in\H_1$ for each fixed $z$ such that
\begin{equation}
 \B_1 f_1(z) = \int_{\R} A_1(z,x)f_1(x)\,\mathrm{d}x.
\end{equation}

\noindent Similar is true for $\B_2 f_2$. The commutative diagrams in conjunction with the above representation for $\B_1$ and $\B_2$ can be realized as
\begin{align}
 \frac{1}{z^{n-1}}\d{\mathrm{d}}{z} \int_{\R} A_1(z,x) f_1(x)\,\mathrm{d}x &= \int_{\R} A_2(z,x) \frac{1}{\sqrt{2}}\bigg(\frac{1}{x^{n-1}}\d{\mathrm{d}}{x} + x^n\bigg) f_1(x)\,\mathrm{d}x \\
 z^n \int_{\R} A_1(z,x) f_1(x)\,\mathrm{d}x &= \int_{\R} A_2(z,x) \frac{1}{\sqrt{2}}\bigg(-\frac{1}{x^{n-1}}\d{\mathrm{d}}{x} + x^n\bigg) f_1(x)\,\mathrm{d}x \\
 z^n \int_{\R} A_2(z,x) f_2(x)\,\mathrm{d}x &= \int_{\R} A_1(z,x) \frac{1}{\sqrt{2}}\bigg(-\d{\mathrm{d}}{x}\frac{1}{x^{n-1}} + x^n\bigg) f_2(x)\,\mathrm{d}x \\
 \d{\mathrm{d}}{z}\frac{1}{z^{n-1}} \int_{\R} A_2(z,x) f_2(x)\,\mathrm{d}x &= \int_{\R} A_1(z,x) \frac{1}{\sqrt{2}}\bigg(\d{\mathrm{d}}{x}\frac{1}{x^{n-1}} + x^n\bigg) f_2(x)\,\mathrm{d}x.
\end{align}

Since we wish for these identities to hold for all $f_1$ and $f_2$, applying integrations by parts---assuming sufficiently nice (exponentially decaying) behavior for $A_1$ and $A_2$---that the following relations hold:
\begin{align}
 \frac{1}{z^{n-1}}\d{\partial}{z} A_1(z,x) &= \frac{1}{\sqrt{2}} \bigg(-\d{\partial}{x}\frac{1}{x^{n-1}} + x^n\bigg) A_2(z,x) \\
 \d{\partial}{z}\frac{1}{z^{n-1}} A_2(z,x) &= \frac{1}{\sqrt{2}}\bigg(-\frac{1}{x^{n-1}}\d{\partial}{x} + x^n\bigg) A_1(z,x) \\
 z^n A_1(z,x) &= \frac{1}{\sqrt{2}} \bigg(\d{\partial}{x}\frac{1}{x^{n-1}} + x^n\bigg) A_2(z,x) \label{eq:A1} \\
 z^n A_2(z,x) &= \frac{1}{\sqrt{2}} \bigg(\frac{1}{x^{n-1}}\d{\partial}{x} + x^n\bigg) A_1(z,x) \label{eq:A2}
\end{align}

\noindent By taking linear combinations, we obtain
\begin{align}
\frac{1}{\sqrt{2}}\bigg(\frac{1}{z^{n-1}}\d{\partial}{z} + z^n\bigg)A_1(z,x) &= x^n A_2(z,x) \label{eq:A3} \\
\frac{1}{\sqrt{2}}\bigg(\d{\partial}{z} \frac{1}{z^{n-1}} + z^n\bigg)A_2(z,x) &= x^n A_1(z,x) \label{eq:A4}
\end{align}

\noindent We must first note a pair of identities that will aid in the solution of the above partial differential equations.

\begin{proposition}
The following basic ``integrating factor''-like identities hold.

\begin{align}
 \bigg(\frac{1}{t^{n-1}}\d{\mathrm{d}}{t} + t^n\bigg)e^{-\frac{t^{2n}}{2n}}f(t) &= e^{-\frac{t^{2n}}{2n}}\frac{1}{t^{n-1}} \d{\mathrm{d}}{t}f(t), \\
 \bigg(\d{\mathrm{d}}{t}\frac{1}{t^{n-1}}+t^n\bigg)e^{-\frac{t^{2n}}{2n}}f(t) &= e^{-\frac{t^{2n}}{2n}} \d{\mathrm{d}}{t} \frac{1}{t^{n-1}} f(t).
\end{align}
\end{proposition}

\noindent Noting these identities and the symmetry in $x$ and $z$ in \eqref{eq:A1}--\eqref{eq:A4}, define $B_1(z,x) = e^{\frac{z^{2n}}{2n}} A_1(z,x) e^{\frac{x^{2n}}{2n}}$ and $B_2(z,x) = e^{\frac{z^{2n}}{2n}}A_2(z,x)e^{\frac{x^{2n}}{2n}}$. $B_1$ and $B_2$ then solve simpler coupled partial differential equations:
\begin{align}
\frac{1}{x^{n-1}} \d{\partial}{x} B_1(z,x) &= \sqrt{2} z^n B_2(z,x) \\
\frac{1}{z^{n-1}} \d{\partial}{z} B_1(z,x) &= \sqrt{2} x^n B_2(z,x) \\
\d{\partial}{x} \frac{1}{x^{n-1}} B_2(z,x) &= \sqrt{2} z^n B_1(z,x) \\
\d{\partial}{z} \frac{1}{z^{n-1}} B_2(z,x) &= \sqrt{2} x^n B_1(z,x).
\end{align}

These equations are again symmetric in $x$ and $z$ and resemble the intertwining relations for $\rho_1$ and $\rho_2$ with the primary differences being the appearance of the multiplicative factor of $\sqrt{2}$ and $x$ rather than $\bar{z}$. Taking cues from the analysis for $\rho_1$ and $\rho_2$, $B_1$ and $B_2$ have the following forms
\begin{align}
B_1(z,x) &= \alpha \sum_{l=0}^{\infty} \frac{2^l(zx)^{2nl}}{(2n)^{2l} \Gamma\big(l+\frac{1}{2n}\big)l!} + \beta \sum_{l=0}^{\infty} \frac{2^{l+\frac{1}{2}} (zx)^{2nl+2n-1}}{(2n)^{2l+1}\Gamma(l+2-\frac{1}{2n}\big)l!}, \\
B_2(z,x) &= \beta \sum_{l=0}^{\infty} \frac{2^l (zx)^{2nl+n-1}}{(2n)^{2l}\Gamma\big(l+1-\frac{1}{2n}\big)l!} + \alpha \sum_{l=0}^{\infty} \frac{2^{l+\frac{1}{2}}(zx)^{2nl+n}}{(2n)^{2l+1}\Gamma(l+1+\frac{1}{2n}\big)l!},
\end{align}

\noindent where $\alpha$ and $\beta$ are to be determined. To uniquely identify $\alpha$ and $\beta$ (which will also give unitarity as we will see shortly), we require that
\begin{align}
e_0(z) &= \int_{\R} A_1(z,x) \psi_0(x)\,\mathrm{d}x \label{eq:unitary1} \\
\widetilde{e}_0(z) &= \int_{\R} A_2(z,x) \widetilde{\psi}_0(x)\,\mathrm{d}x \label{eq:unitary2}
\end{align}

\noindent Under these assumptions, $\alpha$ and $\beta$ are given by
\begin{align}
\alpha &= n^{\frac{1}{2}-\frac{1}{4n}} \sqrt{\Gamma\bigg(\frac{1}{2n}\bigg)} \\
\beta &= 2^{-\frac{1}{2} + \frac{1}{2n}} n^{-\frac{1}{2} + \frac{3}{4n}} \sqrt{\Gamma\bigg(\frac{1}{2n}\bigg)}.
\end{align}

\noindent Simply, to determine $\alpha$, plug in $z=0$ on both sides and evaluate the integral; to determine $\beta$, divide both sides by $z^{n-1}$, plug in $z=0$, and evaluate the integral.

As before with $F_w$ and $\widetilde{F}_w$, representations in terms of the hypergeometric function ${}_0F_1$ exist for $B_1$ and $B_2$ and thus $A_1$ and $A_2$:
\begin{align}
B_1(z,x) &= \frac{\alpha}{\Gamma\big(\frac{1}{2n}\big)} {}_0F_1\bigg(;\frac{1}{2n};\frac{(zx)^{2n}}{2n^2}\bigg) + \frac{\beta}{\sqrt{2}n\Gamma\big(2-\frac{1}{2n}\big)} (zx)^{2n-1} {}_0F_1\bigg(;2-\frac{1}{2n}; \frac{(zx)^{2n}}{2n^2}\bigg), \\
B_2(z,x) &= \frac{\beta}{\Gamma\big(1-\frac{1}{2n}\big)} (zx)^{n-1} {}_0F_1\bigg(;1-\frac{1}{2n}; \frac{(zx)^{2n}}{2n^2}\bigg) + \frac{\sqrt{2}\alpha}{\Gamma\big(\frac{1}{2n}\big)} (zx)^n {}_0F_1\bigg(;1+\frac{1}{2n};\frac{(zx)^{2n}}{2n^2}\bigg).
\end{align}

\begin{remark}
When $n=1$, $\alpha = \pi^{\frac{1}{4}} = \beta$ and $\B_1$ and $\B_2$ simplify to a single integral operator given by
\begin{equation}
\B f(z) = \frac{1}{\pi^{\frac{1}{4}}} \int_{\R} e^{-\frac{z^2}{2}} e^{\sqrt{2}zx} e^{-\frac{x^2}{2}} f(x)\,\mathrm{d}x
\end{equation}

\noindent which can be recognized as the usual Segal-Bargmann transform.
\end{remark}

\begin{remark}
Of particular note is the appearance of the generalized Gaussians $e^{-\frac{x^{2n}}{2n}}$ and $e^{-\frac{z^{2n}}{2n}}$ in the coupled SUSY Segal-Bargmann transform integral kernels. In the traditional Segal-Bargmann setting, the appearance of the Gaussians $e^{-\frac{z^2}{2}}$ and $e^{-\frac{x^2}{2}}$ and the exponential $e^{\sqrt{2}zx}$ in the traditional Segal-Bargmann transform integral kernel seem like a happy accident; however in this general setting, it becomes clear that these are intrinsic and emerge from the ``integrating factor''-like simplification of the ladder operators. Furthermore, the connection between the weight $\rho$ and the exponential is clear in this setting: both obey very similar differential equations.
\end{remark}

\begin{theorem}
The Segal-Bargmann transforms $\B_1:\H_1\to\F_1$ and $\B_2:\H_2\to\F_2$ are unitary, i.e. for $f_1\in \H_1$ and $f_2\in\H_2$,
\begin{align}
\langle \B_1 f_1, \B_1 f_1\rangle_1 &= \langle f_1,f_1\rangle, \\
\langle \B_2 f_2, \B_2 f_2\rangle_2 &= \langle f_2, f_2\rangle.
\end{align}
\end{theorem}

\begin{proof}
By construction of $A_1$ and $A_2$,
\begin{align}
\B_1 \psi_0 &= e_0, \\
\B_2 \widetilde{\psi}_0 &= \widetilde{e}_0.
\end{align}

\noindent Furthermore, a direct computation shows that $\B_1 \psi_1 = e_1$. This can be shown by noting that $a\psi_1 = \sqrt{2n-1} \widetilde{\psi}_0$ and using the coupled SUSY Segal-Bargmann transform relations to get that $\a \B_1 \psi_1 = \B_2 a\psi_1 = \sqrt{2n-1}\B_2 \widetilde{\psi}_0$. Solving the differential equation $\frac{1}{z^{n-1}}\d{\mathrm{d}}{z}\B_1\psi_1 = \sqrt{2n-1}\widetilde{e}_0$ gives that $\B_1 \psi_1 = e_1$. Showing $\B_2 \widetilde{\psi}_1 = \widetilde{e}_1$ proceeds similarly.

By successively applying raising operators and using these four base cases, inductively it can be shown more generally that
\begin{align*}
\B_1 \psi_l &= e_l, \\
\B_2 \widetilde{\psi}_l &= \widetilde{e}_l.
\end{align*}

\noindent This follows from the observation that $a^*b^*\psi_l \propto \psi_{l+2}$. The proportionality constant is determined by the eigenvalue exactly, and both representations (real line and holomorphic representations) of the coupled SUSY in question have the same spectra as they are both unbroken and share the same $\gamma$ and $\delta$.

Thus $\B_1$ and $\B_2$ are surjective and norm-preserving as all $\psi_l$, $\widetilde{\psi}_l$, $e_l$, and $\widetilde{e}_l$ are normalized and form orthonormal bases for their respective spaces, and therefore $\B_1$ and $\B_2$ are unitary.
\end{proof}

Much like the typical Segal-Bargmann transform, we also have explicit expressions for the inverse coupled SUSY Segal-Bargmann transforms.

\begin{corollary}
The coupled SUSY Segal-Bargmann transforms $\B_1$ and $\B_2$ have inverses $\B_1^{-1}:\F_1\to\H_1$ and $\B_2^{-1}:\F_2\to\H_2$, respectively, given by
\begin{align}
\B_1^{-1} f_1(x) &= \int_{\C} A_1(\bar{z},x) f_1(z)\rho_1(z,\bar{z})\,\mathrm{d}A(z) \\
\B_2^{-1} f_2(x) &= \int_{\C} A_2(\bar{z},x) f_2(z)\rho_2(z,\bar{z})\,\mathrm{d}A(z)
\end{align}
\end{corollary}

\begin{proof}
The proof follows from a simple application of Fubini-Tonelli, noting the exponential decay of $A_1$ and $A_2$ as well as $\rho_1$ and $\rho_2$ or by acting directly on the basis elements $e_l$ and $\widetilde{e}_l$ and extending linearly.
\end{proof}

\begin{remark}
A quick inspection shows that there is a canonical inclusion of the Segal-Bargmann spaces $\F_1$ and $\F_2$ into $\F$ by sending $z^k$ in $\F_1$ and $\F_2$ to $z^k$ in $\F$. A function in $L^2(\R,\mathrm{d}x)$ will in general have distinct holomorphic representations through different coupled SUSY Segal-Bargmann transforms, not in contradiction with the identity theorem as the representations are related by the Segal-Bargmann transforms and are not directly equal to each other.
\end{remark}

\section{The Connection to Short-Time Transforms and Coherent States} \label{sec:5}

The usual Segal-Bargmann transform is closely related to the short-time Fourier transform and coherent states. The short-time Fourier transform is a time-frequency representation wherein the frequency content of segment of a function is analyzed \cite{grochenig}. This is achieved by taking the Fourier transform of the function with a moving window. Rigorously,

\begin{definition}
Given (typically non-negative and normalized) $g \in L^2(\R,\mathrm{d}x)$, the short-time Fourier transform with window function $g$ is the map $\mathcal{V}_g:L^2(\R,\mathrm{d}x)\to L^2(\R^2,\mathrm{d}x\,\mathrm{d}y)$ given by
\begin{equation}
\mathcal{V}_gf(\omega, t) = \frac{1}{\sqrt{2\pi}} \int_{\R} e^{-i\omega \tau} g(\tau-t) f(\tau)\,\mathrm{d}\tau.
\end{equation}
\end{definition}

This is a well-defined integral operator as $g_tf \in L^1(\R,\mathrm{d}x)$ by Cauchy-Schwartz. That it maps to $L^2(\R^2,\mathrm{d}x\,\mathrm{d}y)$ follows from the unitarity of the Fourier transform on $L^2(\R,\mathrm{d}x)$.

Coherent states have multiple different definitions that are not in general equivalent \cite{brif}. At least three competing definitions exist: a coherent state is an eigenfunction of the lowering operator, coherent states are generated from dilation operators, and coherent states are minimum uncertainty states (for a given pair of non-commuting operators). In some cases, these definitions can coincide, for instance in the oscillator algebra case. For our purposes, we will be concerned with eigenfunctions of a lowering operator.

The connection between the short-time Fourier transform and the Segal-Bargmann transform can be made by letting $z = \frac{t}{\sqrt{2}} - i\sqrt{2}\omega$ and $g(t) = \frac{1}{\pi^{\frac{1}{4}}} e^{-\frac{t^2}{2}}$ to get
\begin{align}
\mathcal{V}_g f(\omega, t) = \frac{1}{\sqrt{2\pi}} e^{-\frac{t^2}{4} - \omega^2} e^{2i\omega t} \B f(z).
\end{align}

\noindent The connection between coherent states and the Segal-Bargmann transform is immediate as the Segal-Bargmann kernel is an eigenfunction of the lowering operator.

In \cite{williams1}, the authors defined a short-time analogue for the $\Phi_n$ transforms built upon the concept of moving windows. The $\Phi_n$ transform of a sufficiently nice function as an integral transform is given by
\begin{equation}
\Phi_n f(y) = \int_{\R} \varphi_n(xy) f(x)\,\mathrm{d}x,
\end{equation}

\noindent where $\varphi_n$ is a solution to the differential equation
\begin{equation}
-\d{\textrm{d}}{x}\frac{1}{x^{2n-2}}\d{\textrm{d}}{x}\varphi_n(xy) = y^{2n}\varphi_n(xy)
\end{equation}

\noindent with $\varphi_n(0) = \frac{n}{(2n)^{\frac{1}{2n}}\Gamma\big(\frac{1}{2n}\big)}$ and $\varphi_n'(0) = -i\frac{n}{(2n)^{2-\frac{1}{2n}}\Gamma\big(2-\frac{1}{2n}\big)}$. The short-time $\Phi_n$ transform of a function $f$ with window $g$ is then defined to be
\begin{equation}
\mathcal{V}_g^{(n)}(y,x) = \int_{\R} \varphi_n(y(x'-x))g(x'-x)f(x')\,\mathrm{d}x'.
\end{equation}

The $\Phi_n$ transforms are closely related to the coupled SUSY involving $a$ and $b$ by the following identity:
\begin{equation}
\Phi_n ab = -ab \Phi_n.
\end{equation}

\noindent which can be confirmed from a simple integration by parts. However despite connection to the coupled SUSYs at hand, a simple inspection shows that the short-time $\Phi_n$ transforms are distinct from the coupled SUSY Segal-Bargmann transforms. This is a break from traditional Segal-Bargmann transform theory and is a direct result of the fact that the $\Phi_n$ transforms do not play nicely with translations in general---this is a (nearly) uniquely identifying property of the Fourier transform. Thus the coupled SUSY Segal-Bargmann transforms may be viewed as an altogether new time-frequency representation.

Despite the break from short-time $\Phi_n$ transform, the coupled SUSY Segal-Bargmann kernels can be viewed as coherent states and thus the coupled SUSY Segal-Bargmann transforms can be viewed as coherent state transforms. Rewriting \eqref{eq:A2} and \eqref{eq:A3} in a matricial form, we have that
\begin{equation}
\begin{pmatrix} 0 & b \\ a & 0 \end{pmatrix} \begin{pmatrix} A_1 \\ A_2\end{pmatrix} = z^n \begin{pmatrix} A_1 \\ A_2 \end{pmatrix},
\end{equation}

\noindent or by squaring,
\begin{equation}
\begin{pmatrix} ba & 0 \\ 0 & ab \end{pmatrix} \begin{pmatrix} A_1 \\ A_2\end{pmatrix} = z^{2n} \begin{pmatrix} A_1 \\ A_2 \end{pmatrix}.
\end{equation}

\noindent Thus the combined state $\begin{pmatrix} A_1 \\ A_2 \end{pmatrix}$ is an eigenfunction of a lowering operator (for a matricial Hamiltonian in this case) and thus a generalized $\mathfrak{su}(1,1)$ coherent state.

This work is closely related to the coherent states and Segal-Bargmann transform developed in \cite{barut}. Therein, the Segal-Bargmann transform also takes the form of the hypergeometric function ${}_0F_1$ and the weight functions and reproducing kernels are similar. The Segal-Bargmann space is also spanned by $z^l$ for $l=0,1,2,\ldots$, whereas the coupled SUSY Segal-Bargmann spaces developed herein are stricter vector subspaces.

Furthermore, in \cite{barut}, limits were taken to obtain the Heisenberg-Weyl algebra, imagining the $\mathfrak{su}(1,1)$ Lie algebra as a $q$-deformation of the Heisenberg-Weyl algebra, and thus the usual Segal-Bargmann transform is not directly generalized, whereas the quantum harmonic oscillator is baked directly into coupled SUSY and so direct generalizations exist in our case.

Moreover, the representations for the $\mathfrak{su}(1,1)$ elements in \cite{barut} were of a relatively complicated form, particularly the lowering operator, but their analogues in $\a$ and $\b$ and their adjoints and products are of a very simple form in $\F_1$ and $\F_2$. The work presented herein is also of a supersymmetric nature which is not present in their analysis.

The work presented herein elucidates some of the traditional Segal-Bargmann transform theory. In the traditional setting, due to properties of exponentials, there is some symmetry between the weight $\rho$ and the integral kernel $A$ which obfuscates some of the mathematics. In this general setting, we see that this symmetry no longer exists. However, when considering the asymptotic forms of the weights $\rho_1$ and $\rho_2$, some of the aforementioned symmetry reveals itself again.

\section*{Acknowledgements}

The author would like to thank Dr. John R. Klauder for the suggestion for this avenue of research and the late Dr. Donald J. Kouri for his mentorship in this project.

\bibliographystyle{plain}
\bibliography{bargmann_transforms_ref}

\end{document}